\documentclass[11pt]{amsart}

\usepackage[T1]{fontenc}
\usepackage{lmodern}
\usepackage{microtype}
\usepackage{needspace}
\usepackage{amsmath,amssymb,amsthm,mathtools}
\usepackage[colorlinks=true,linkcolor=blue,citecolor=blue,urlcolor=blue]{hyperref}
\usepackage[nameinlink,noabbrev]{cleveref}
\hypersetup{
  pdftitle={Endpoint Riesz transforms on normalized graphs},
  pdfauthor={Author Name},
  pdfsubject={Weak-type estimates for edge and vertex Riesz transforms},
  pdfkeywords={normalized graph Laplacian, Riesz transform, obstacle decomposition}
}
\allowdisplaybreaks[2]

\newtheorem{theorem}{Theorem}[section]
\newtheorem{proposition}[theorem]{Proposition}
\newtheorem{lemma}[theorem]{Lemma}
\newtheorem{corollary}[theorem]{Corollary}
\crefname{theorem}{Theorem}{Theorems}
\crefname{proposition}{Proposition}{Propositions}
\crefname{lemma}{Lemma}{Lemmas}
\crefname{corollary}{Corollary}{Corollaries}

\newcommand{\norm}[1]{\lVert #1\rVert}

\newcommand{\cH}{\mathcal H}
\newcommand{\cR}{\mathcal R}

\title[Endpoint Riesz transforms on normalized graphs]
{Endpoint Riesz transforms on normalized graphs}

\author{Zihao Wang}
\address{School of Mathematical Sciences, Fudan University, Shanghai 200433, China}
\email{wangzh25@m.fudan.edu.cn}
\date{September 23, 2026}


\begin{document}

\begin{abstract}
For an infinite connected locally finite undirected graph with positive
symmetric edge weights and normalized
Laplacian $\Delta$, we prove weak type $(1,1)$ estimates for two Riesz
transforms. The edge transform has universal constant $2$, with no volume
assumption. The vertex carr\'e-du-champ transform has constant $H_1+1$
whenever $H_1=\sup_x m(B(x,1))/m(x)<\infty$. These endpoint constants apply
to real-valued inputs and yield strong $\ell^p$ bounds for $1<p\le2$.
The proof uses a discrete obstacle decomposition obtained by a contraction
mapping, followed by elementary support estimates for the two gradients.
\end{abstract}

\maketitle

\section{Introduction}

Riesz transforms connect first-order differential estimates with the
functional calculus of a nonnegative Laplacian. Their $L^p$ theory on
manifolds has been developed through heat-kernel estimates and
Littlewood--Paley methods; see
\cite{Stein1970,CoulhonDuong1999,CoulhonDuong2003,
CoulhonDuongLi2003,AuscherCoulhonDuongHofmann2004}.
Sub-Gaussian analogues and quasi-Riesz transforms were studied in
\cite{Chen2015,ChenCoulhonFeneuilRuss2017}.

On graphs, the work of Russ \cite{Russ2000,Russ2001}, Dungey
\cite{Dungey2004,Dungey2008}, and Badr--Russ \cite{BadrRuss2009}
established Riesz-transform, Hardy-space, and square-function estimates
in several geometric settings. Heat-kernel bounds for reversible walks
and their relation to volume growth and Poincar\'e inequalities form an
important part of this theory; see
\cite{HebischSaloffCoste1993,Delmotte1999}. Background on normalized
Laplacians, random walks, and graph Dirichlet forms can be found in
\cite{Chung1997,Woess2000,KellerLenz2012}.

A distinction specific to the discrete setting is the choice of gradient.
The edge difference is measured on the edge space, whereas the vertex
carr\'e-du-champ gradient combines neighboring differences in a Hilbert
fiber over each vertex. These gradients have the same $\ell^2$ energy,
but their $\ell^p$ norms need not be comparable for $p\ne2$.
Chen--Coulhon--Hua \cite{ChenCoulhonHua2020} use this distinction to obtain
edge-gradient semigroup and square-function estimates, and to construct
counterexamples for the vertex gradient when local control is absent.

The obstacle method provides a different route to endpoint estimates.
Ouyang--Spector--Stockdale \cite{OuyangSpectorStockdale2026} use a
fractional obstacle problem for the Euclidean vector Riesz transform.
In a recent preprint, Chen--Jiang--Li--Li \cite{ChenJiangLiLi2026}
obtain endpoint estimates on complete noncompact manifolds and local
Hilbertian Dirichlet spaces by an abstract obstacle decomposition. Their localization step uses the fact that a Sobolev
differential vanishes almost everywhere on a level set. Graph gradients
do not have this property: a zero value at one vertex does not prevent
a nonzero difference across an incident edge.

In this paper we prove endpoint estimates for both graph gradients by a
direct discrete argument. The edge transform requires no volume condition.
For the vertex transform, we assume only that
$\sup_x m(B(x,1))/m(x)$ is finite; no global volume doubling is required.
The obstacle decomposition follows from
\[
 \Delta^{1/2}=I-Q,
 \qquad Q=\sum_{k\ge1}c_kP^k,
 \qquad c_k>0,
 \qquad \sum_{k\ge1}c_k=1.
\]
For nonnegative real-valued $f$ and $\kappa>0$, positivity and contractivity
of $Q$ allow us to solve
$u=(1+\kappa)^{-1}(Qu+f-\lambda)_+$ by a contraction mapping.
The resulting potential is supported on a set of controlled vertex
measure. For the edge gradient, we control the measure of edges incident
to this set; for the vertex gradient, we control its one-step neighborhood.
These support estimates, together with the $\ell^2$ energy identity,
give the two weak-type bounds. The proof uses no heat-kernel estimates
or discrete-time analyticity assumption.

\section{Setting and main results}\label{sec:setting}

\subsection{The normalized Laplacian}
Let $G=(V,E,\mu)$ be an infinite connected locally finite undirected graph
without loops or multiple edges. Thus every vertex has finitely many
neighbors, and any two vertices can be joined by a finite edge path.
Write $x\sim y$ when $\{x,y\}\in E$. Assume
\[
 \mu_{xy}=\mu_{yx}\in(0,\infty)\quad(x\sim y),
 \qquad \mu_{xy}=0\quad(x\not\sim y),
 \qquad m(x)=\sum_{y\sim x}\mu_{xy}.
\]
Every combinatorial ball of finite radius is finite. Taking their union
about one vertex shows that $V$, and hence $E$, is countable.
Connectedness gives $0<m(x)<\infty$. For $F\subset V$, set
$m(F)=\sum_{x\in F}m(x)$. No lower bound on $m(x)$, and no assumption that
$m(V)=\infty$, is imposed.

For a countable set $Z$ with weights $a(z)\in(0,\infty)$, write
\[
 \ell^r(Z,a)=\left\{h:\sum_{z\in Z}|h(z)|^r a(z)<\infty\right\},
 \qquad 1\le r<\infty,
\]
with norm $\norm h_{r,a}=(\sum_z|h(z)|^r a(z))^{1/r}$.
The space $\ell^\infty(Z,a)$ consists of the bounded functions, with
$\norm h_{\infty}=\sup_z|h(z)|$. On $\ell^2(Z,a)$ we use
$\langle h,k\rangle_{2,a}=\sum_z h(z)\overline{k(z)}a(z)$.
For $Z=V$, let $c_{00}(V)$ denote the finitely supported functions.
They are dense in $\ell^r(V,m)$ for every $r<\infty$, by truncation along
an increasing exhaustion of $V$ by finite sets.
Unless specified otherwise, the Hilbert spaces below are complex; all
order inequalities and obstacle decompositions concern real-valued functions.

Define
\[
 P(x,y)=\frac{\mu_{xy}}{m(x)},\qquad
 Pf(x)=\sum_yP(x,y)f(y),\qquad \Delta=I-P.
\]
These pointwise sums are finite. We have
\begin{equation}\label{eq:reversibility}
 \begin{gathered}
 \sum_yP(x,y)=1,\qquad
 m(x)P(x,y)=\mu_{xy}=m(y)P(y,x),\\
 \sum_xm(x)P(x,y)=m(y).
 \end{gathered}
\end{equation}
In particular, $P$ preserves nonnegativity. For $1\le r<\infty$, Jensen's
inequality and Tonelli's theorem give
\[
 \sum_x|Ph(x)|^r m(x)
 \le\sum_{x,y}P(x,y)|h(y)|^r m(x)
 =\sum_y|h(y)|^r m(y).
\]
The $\ell^\infty$ contraction follows from the first identity in
\eqref{eq:reversibility}. For $h\in\ell^1(V,m)$,
\[
 \sum_{x,y}m(x)P(x,y)|h(y)|=\norm h_{1,m}<\infty,
\]
so Fubini's theorem yields
\begin{equation}\label{eq:mass-P}
 \sum_xPh(x)m(x)=\sum_xh(x)m(x).
\end{equation}
Reversibility also gives $\langle Pf,h\rangle_{2,m}
=\langle f,Ph\rangle_{2,m}$ first on $c_{00}(V)$, and then on $\ell^2$
by continuity. Thus $P$ is a self-adjoint contraction,
$\sigma(P)\subset[-1,1]$, and $\Delta$ is a bounded nonnegative
self-adjoint operator with domain $\ell^2(V,m)$ and $0\le\Delta\le2I$.

\subsection{The two gradients and their Riesz transforms}
Fix one orientation of each edge, identify $E$ with this oriented set, and put
\[
 D_Ef(e)=f(x)-f(y),\qquad \mu_E(e)=\mu_{xy}\quad(e=(x,y)\in E).
\]
Each unoriented edge is counted exactly once in the measure
$\mu_E(A)=\sum_{e\in A}\mu_E(e)$. Reversing an orientation changes only a
sign. The edge differential takes values in $\ell^2(E,\mu_E)$.

At a vertex $x$, define the finite-dimensional Hilbert space
\[
 \cH_x=\left\{\xi:\{y:y\sim x\}\longrightarrow\mathbb C\right\},
 \qquad
 \langle\xi,\eta\rangle_{\cH_x}
   =\sum_{y\sim x}P(x,y)\xi(y)\overline{\eta(y)}.
\]
The vertex differential is the section of $\cH=(\cH_x)_{x\in V}$ given by
\[
 (D_Vf)(x)(y)=\frac{f(x)-f(y)}{\sqrt2}\quad(y\sim x).
\]
Its pointwise norm is
\begin{equation}\label{eq:vertex-gradient}
 |D_Vf(x)|_{\cH_x}
  =\left(\frac1{2m(x)}\sum_{y\sim x}\mu_{xy}|f(x)-f(y)|^2\right)^{1/2}.
\end{equation}
For $1\le r<\infty$, let $\ell^r(V,m;\cH)$ be the space of sections $F$
with $\norm F_{r,m;\cH}^r=\sum_x|F(x)|_{\cH_x}^r m(x)<\infty$; for
$r=\infty$, use $\sup_x|F(x)|_{\cH_x}$.

For $f\in\ell^2(V,m)$, the inequality
$|f(x)-f(y)|^2\le2|f(x)|^2+2|f(y)|^2$ makes the following sums finite:
\begin{equation}\label{eq:energy}
 \norm{D_Ef}_{2,\mu_E}^2
 =\norm{D_Vf}_{2,m;\cH}^2
 =\frac12\sum_{x,y}\mu_{xy}|f(x)-f(y)|^2
 =\langle\Delta f,f\rangle_{2,m}.
\end{equation}
The identity follows first for $c_{00}(V)$ and then by density; alternatively,
all cross terms are absolutely summable by Cauchy--Schwarz.
Consequently $D_E$ and $D_V$ are bounded operators on their full input
space $\ell^2(V,m)$, of norm at most $\sqrt2$, and
$D_E^*D_E=D_V^*D_V=\Delta$.

Let $E_\Delta$ be the spectral resolution of $\Delta$, set
$\Pi_0=E_\Delta(\{0\})$, and write
$\rho_f(B)=\norm{E_\Delta(B)f}_{2,m}^2$ for Borel sets $B\subset[0,2]$.
The positive square root $\Delta^{1/2}$ is bounded on $\ell^2$.
The generalized inverse square root is defined by
\[
 \Delta^{-1/2}f=\int_{(0,2]}t^{-1/2}\,dE_\Delta(t)f,
\]
on the domain
\[
 \operatorname{Dom}(\Delta^{-1/2})
 =\left\{f\in\ell^2:\int_{(0,2]}t^{-1}\,d\rho_f(t)<\infty\right\};
\]
it is zero on $\ker\Delta$. To define a Riesz transform on all of
$\ell^2$, use the polar decomposition, for $D\in\{D_E,D_V\}$:
\begin{equation}\label{eq:polar}
 D=\cR_D\Delta^{1/2},\qquad
 \cR_D|_{\ker\Delta}=0,\qquad
 \norm{\cR_D f}_2=\norm{(I-\Pi_0)f}_{2,m}.
\end{equation}
Here $\cR_D$ is the partial isometry with initial space
$(\ker\Delta)^\perp$ and final space $\overline{\operatorname{Ran}D}$.
It agrees with $D\Delta^{-1/2}$ on $\operatorname{Dom}(\Delta^{-1/2})$;
throughout the paper the latter notation denotes this bounded polar
extension. Set $\cR_E=\cR_{D_E}$ and $\cR_V=\cR_{D_V}$.
The energy identity and connectedness show that $\ker\Delta$ consists of
the constant functions belonging to $\ell^2(V,m)$: it is zero if
$m(V)=\infty$, and is spanned by $1$ otherwise. Thus no mean-zero
restriction or spectral-gap assumption is needed.

For either scalar functions or Hilbert-valued sections on a weighted
countable set $(Z,a)$, define
\[
 \norm F_{\ell^{1,\infty}(Z,a)}
   =\sup_{t>0}t\,a\{z:|F(z)|>t\}.
\]
The weak space consists of the functions or sections for which this
quasi-norm is finite. For sections, $|F(z)|$ is the norm in the fiber at $z$.
An operator is of weak type $(1,1)$ with constant $C$ if this quantity
is at most $C$ times the input $\ell^1$ norm.

\Needspace{7\baselineskip}
\subsection{Endpoint estimates}
\begin{theorem}[Edge transform]\label{thm:edge}
For every real-valued $f\in\ell^1(V,m)\cap\ell^2(V,m)$ and every
$\lambda>0$,
\begin{equation}\label{eq:edge-main}
 \mu_E\{e:|\cR_Ef(e)|>\lambda\}
 \le\frac2\lambda\norm f_{1,m}.
\end{equation}
The transform has a unique bounded extension
$\ell^1(V,m;\mathbb R)\to\ell^{1,\infty}(E,\mu_E;\mathbb R)$ with the
same estimate. For $1<p\le2$, it extends boundedly from $\ell^p(V,m)$
to $\ell^p(E,\mu_E)$ with norm bounded by a constant depending only on $p$.
\end{theorem}

Let $\operatorname{dist}(x,y)$ be the least number of edges in a path
joining $x$ to $y$. For $F\ne\varnothing$, put
$\operatorname{dist}(x,F)=\inf_{y\in F}\operatorname{dist}(x,y)$, and set
\[
 N_1(F)=\{x:\operatorname{dist}(x,F)\le1\},\qquad
 N_1(\varnothing)=\varnothing,\qquad B(x,1)=N_1(\{x\}).
\]
Define the possibly infinite constant
\begin{equation}\label{eq:H1}
 H_1(G)=\sup_{x\in V}\frac{m(B(x,1))}{m(x)}.
\end{equation}
It also satisfies
\begin{equation}\label{eq:H1-sets}
 H_1(G)=\sup_{\substack{F\subset V\\0<m(F)<\infty}}
                \frac{m(N_1(F))}{m(F)}.
\end{equation}
Singletons give one inequality. For the other, countable subadditivity gives
\[
 m(N_1(F))\le\sum_{x\in F}m(B(x,1))\le H_1(G)m(F)
\]
when $H_1(G)<\infty$; if $H_1(G)=\infty$, the singleton lower bound
already proves \eqref{eq:H1-sets}.
As there are no loops,
$m(B(x,1))=m(x)+\sum_{y\sim x}m(y)$.
Thus $H_1(G)<\infty$ is exactly the one-step local doubling condition
$\sum_{y\sim x}m(y)\le C m(x)$ used in
\cite{Dungey2008,ChenCoulhonHua2020}.

\begin{theorem}[Vertex transform]\label{thm:vertex}
Suppose that $H_1(G)<\infty$. For every real-valued
$f\in\ell^1(V,m)\cap\ell^2(V,m)$ and every $\lambda>0$,
\begin{equation}\label{eq:vertex-main}
 m\{x:|\cR_Vf(x)|_{\cH_x}>\lambda\}
 \le\frac{H_1(G)+1}{\lambda}\norm f_{1,m}.
\end{equation}
The transform has a unique bounded extension from real $\ell^1(V,m)$
to the real sections of $\ell^{1,\infty}(V,m;\cH)$ with the same estimate.
For $1<p\le2$, it extends boundedly from $\ell^p(V,m)$ to
$\ell^p(V,m;\cH)$, with norm bounded in terms of $p$ and $H_1(G)$ only.
\end{theorem}

The stated endpoint constants concern real-valued inputs. The strong
$\ell^p$ conclusions hold also for complex-valued inputs, as proved at
the end of \cref{sec:proofs}. Both $\ell^2$ operator norms are at most $1$.

\section{A discrete obstacle decomposition}\label{sec:obstacle}

\begin{lemma}[Square-root expansion]\label[lemma]{lem:Q}
For $k\ge1$, set
\[
 c_k=\frac1{2k-1}\frac{\binom{2k}{k}}{4^k}.
\]
Then $c_k>0$, $\sum_{k\ge1}c_k=1$, and
\[
 Q=\sum_{k\ge1}c_kP^k
\]
converges in operator norm on each $\ell^r(V,m)$, $1\le r\le\infty$.
These operators agree on the intersections of their domains, preserve
nonnegativity, and are contractions. On $\ell^1$, $Q$ preserves mass;
on $\ell^2$,
\begin{equation}\label{eq:sqrt}
 \Delta^{1/2}=I-Q.
\end{equation}
\end{lemma}

\begin{proof}
The binomial expansion gives
\[
 \sqrt{1-z}=1-\sum_{k\ge1}c_kz^k\qquad(|z|<1).
\]
Letting the real variable $z\uparrow1$, monotone convergence of the
nonnegative terms yields $\sum_kc_k=1$. The Weierstrass test therefore
gives uniform absolute convergence on $[-1,1]$, including $-1$ and $1$;
continuity extends the displayed scalar identity to both endpoints.
For $Q_N=\sum_{k=1}^Nc_kP^k$,
\begin{equation}\label{eq:Q-tail}
 \norm{Q-Q_N}_{\ell^r\to\ell^r}\le\sum_{k>N}c_k\longrightarrow0.
\end{equation}
Indeed, the partial sums are Cauchy in operator norm by the contraction
bounds for $P^k$. Their limits preserve nonnegativity, and have norm at
most $\sum_kc_k=1$. For $r<\infty$, coordinate evaluation satisfies
$|h(x)|\le m(x)^{-1/r}\norm h_{r,m}$; for $r=\infty$ it is also
continuous. Hence the operator limits on any two $\ell^r$ spaces agree
pointwise on their intersection.
For $h\in\ell^1$, the mass functional $h\mapsto\sum_xh(x)m(x)$ is
continuous. Since $P^k$ preserves mass,
\begin{equation}\label{eq:mass-Q}
 \sum_xQh(x)m(x)
 =\lim_{N\to\infty}\left(\sum_{k=1}^Nc_k\right)\sum_xh(x)m(x)
 =\sum_xh(x)m(x).
\end{equation}
Finally, continuous functional calculus on $\sigma(P)\subset[-1,1]$
applied to the uniform scalar expansion gives \eqref{eq:sqrt}.
\end{proof}

Fix $\kappa>0$ and define
\begin{equation}\label{eq:Akappa}
 A_\kappa=\Delta^{1/2}+\kappa I=(1+\kappa)I-Q
 \quad\text{on }\ell^2.
\end{equation}
The last expression defines a compatible bounded operator on every
$\ell^r$. Its inverse exists there by the operator-norm convergent
Neumann series
\begin{equation}\label{eq:Neumann}
 A_\kappa^{-1}
 =\frac1{1+\kappa}\sum_{j=0}^\infty
                 \left(\frac{Q}{1+\kappa}\right)^j,
 \qquad \norm{A_\kappa^{-1}}_{\ell^r\to\ell^r}\le\kappa^{-1}.
\end{equation}
On $\ell^2$, this is also the spectral inverse of
$\Delta^{1/2}+\kappa I$. For a real scalar $s$, write $s_+=\max\{s,0\}$.

\begin{proposition}[Positive decomposition]\label[proposition]{prop:positive}
Let $0\le f\in\ell^1(V,m)\cap\ell^2(V,m)$ and $\lambda,\kappa>0$.
There are $u,b\in\ell^1(V,m)\cap\ell^2(V,m)$ such that
\begin{equation}\label{eq:positive-dec}
 f=A_\kappa u+b,\qquad
 u\ge0,\qquad 0\le b\le\lambda,\qquad
 b=\lambda\ \text{on }\{u>0\}.
\end{equation}
In particular $b\in\ell^\infty$, and
\begin{equation}\label{eq:mass-id}
 \norm b_{1,m}+\kappa\norm u_{1,m}=\norm f_{1,m},
 \qquad \lambda m(\{u>0\})\le\norm f_{1,m}.
\end{equation}
\end{proposition}

\begin{proof}
Use the real Banach space
$X=\ell^1(V,m;\mathbb R)\cap\ell^2(V,m;\mathbb R)$ with norm
$\norm v_X=\norm v_{1,m}+\norm v_{2,m}$.
For completeness, an $X$-Cauchy sequence converges separately in
$\ell^1$ and $\ell^2$; continuous coordinate evaluation identifies the
two limits, giving convergence in $X$. Its positive cone $X_+$ is closed
for the same reason.
On $X_+$ put
\[
 \Phi(v)=\frac1{1+\kappa}(Qv+f-\lambda)_+.
\]
The expression is taken pointwise. It belongs to $X_+$ even if the
constant function $\lambda$ is not integrable, since
\[
 0\le(Qv+f-\lambda)_+\le Qv+f.
\]
The positive-part map is $1$-Lipschitz, so
\[
 \norm{\Phi(v)-\Phi(z)}_X
 \le\frac1{1+\kappa}\norm{Q(v-z)}_X
 \le\frac1{1+\kappa}\norm{v-z}_X.
\]
The contraction mapping theorem yields a unique fixed point $u\in X_+$.
More explicitly, with $u_0=0$ and $u_{j+1}=\Phi(u_j)$,
\begin{equation}\label{eq:iteration-convergence}
 \norm{u-u_j}_X
 \le\frac{(1+\kappa)^{-j}}{\kappa}\norm f_X
 \longrightarrow0.
\end{equation}
This follows from the geometric-series bound for successive iterates
and $\norm{u_1}_X\le(1+\kappa)^{-1}\norm f_X$.
The limit is pointwise at every vertex and satisfies
\begin{equation}\label{eq:fixed-point}
 (1+\kappa)u=(Qu+f-\lambda)_+.
\end{equation}
Define $b=f-A_\kappa u$. Then $b\in X$ and
\[
 b=Qu+f-(Qu+f-\lambda)_+=\min\{Qu+f,\lambda\}.
\]
Thus $0\le b\le\lambda$, and $b=\lambda$ at each vertex where $u>0$.
All terms in $f=(I-Q)u+\kappa u+b$ are absolutely summable with respect
to $m$. Summing this equality and applying \eqref{eq:mass-Q} gives
\[
 \norm f_{1,m}=\kappa\norm u_{1,m}+\norm b_{1,m}.
\]
Since $b=\lambda$ on $\{u>0\}$, this also proves the support bound in
\eqref{eq:mass-id}. No limit in $\kappa$ is taken in this construction.
\end{proof}

\begin{corollary}[Signed decomposition]\label[corollary]{cor:signed}
For every real-valued $f\in\ell^1(V,m)\cap\ell^2(V,m)$ and every
$\lambda,\kappa>0$, there are $g,w\in\ell^1\cap\ell^2$ and
$\Omega\subset V$ such that
\begin{equation}\label{eq:signed-dec}
 f=g+A_\kappa w,
\end{equation}
\begin{equation}\label{eq:good}
 |g|\le\lambda,\qquad
 \norm g_{1,m}\le\norm f_{1,m},\qquad
 \norm g_{2,m}^2\le\lambda\norm f_{1,m},
\end{equation}
and
\begin{equation}\label{eq:support}
 w=0\ \text{on }V\setminus\Omega,\qquad
 m(\Omega)\le\lambda^{-1}\norm f_{1,m}.
\end{equation}
\end{corollary}

\begin{proof}
Let $f^+=\max\{f,0\}$ and $f^-=\max\{-f,0\}$.
Apply \cref{prop:positive} to $f^+$ and $f^-$ with the same height
$\lambda$ and shift $\kappa$, and write
$f^\pm=A_\kappa u_\pm+b_\pm$. Define
\[
 w=u_+-u_-,\qquad g=b_+-b_-,\qquad
 \Omega=\{u_+>0\}\cup\{u_->0\}.
\]
Then \eqref{eq:signed-dec} holds. As both $b_+$ and $b_-$ lie in
$[0,\lambda]$, their difference has absolute value at most $\lambda$.
Also,
\[
 \norm g_{1,m}\le\norm{b_+}_{1,m}+\norm{b_-}_{1,m}
 \le\norm{f^+}_{1,m}+\norm{f^-}_{1,m}=\norm f_{1,m}.
\]
Hence $\norm g_{2,m}^2\le\norm g_\infty\norm g_{1,m}
\le\lambda\norm f_{1,m}$. Finally,
\[
 \lambda m(\Omega)
 \le\lambda m(\{u_+>0\})+\lambda m(\{u_->0\})
 \le\norm f_{1,m},
\]
and $w$ is zero outside $\Omega$ by construction.
\end{proof}

\section{Proofs of the endpoint estimates}\label{sec:proofs}

\subsection{The shifted transforms and their limit}
For $D\in\{D_E,D_V\}$ define $\cR_{D,\kappa}=DA_\kappa^{-1}$ on
$\ell^2(V,m)$. Equation~\eqref{eq:polar} gives
\begin{equation}\label{eq:shifted}
 \cR_{D,\kappa}
 =\cR_D\Delta^{1/2}(\Delta^{1/2}+\kappa I)^{-1},
 \qquad \norm{\cR_{D,\kappa}}_{2\to2}\le1.
\end{equation}
Here and below, an output norm without subscripts uses the corresponding
edge or vertex Hilbert space. The norm bound follows from $0\le\sqrt t/(\sqrt t+\kappa)\le1$ for
$t\in[0,2]$. By the partial-isometry identity in \eqref{eq:polar},
\begin{equation}\label{eq:spectral-error}
 \norm{\cR_{D,\kappa}f-\cR_Df}_2^2
 =\int_{(0,2]}\left(\frac{\kappa}{\sqrt t+\kappa}\right)^2d\rho_f(t)
 \longrightarrow0.
\end{equation}
Dominated convergence applies because $\rho_f$ is a finite measure,
the integrand is at most $1$, and it tends to zero for $t>0$.
The zero spectral subspace is omitted because both transforms vanish
there. This proves strong, not necessarily operator-norm, convergence:
\begin{equation}\label{eq:strong-limit}
 \cR_{D,\kappa}f\longrightarrow\cR_Df
 \quad\text{in the output }\ell^2\text{ space}.
\end{equation}
For an output atom $z$ of measure $a(z)>0$,
\[
 |F(z)|\le a(z)^{-1/2}\norm F_2,
\]
with the fiber norm used in the vertex case. Thus \eqref{eq:strong-limit}
also gives convergence at every edge, or in every vertex fiber, for the
whole family $\kappa\downarrow0$. The real coefficients in
\eqref{eq:Neumann} show that $A_\kappa^{-1}$ preserves real-valued
functions. Since both differentials do so as well, this pointwise limit
also shows that $\cR_E$ and $\cR_V$ preserve the real subspaces.

\subsection{The edge transform}
For $F\subset V$ let
\[
 E(F)=\{\{x,y\}\in E:x\in F\text{ or }y\in F\}.
\]
If $u=0$ on $V\setminus F$, then $D_Eu=0$ on $E\setminus E(F)$.
Moreover, Tonelli's theorem gives
\begin{equation}\label{eq:edge-support}
 \mu_E(E(F))\le\sum_{x\in F}\sum_{y\sim x}\mu_{xy}=m(F).
\end{equation}
An edge with both endpoints in $F$ is counted twice on the right and once
on the left, so this inequality holds also for infinite $F$.

\begin{proof}[Proof of the endpoint bound in \cref{thm:edge}]
Fix real $f\in\ell^1\cap\ell^2$ and $\lambda,\kappa>0$.
Apply \cref{cor:signed} at height $\lambda$ to obtain
$f=g+A_\kappa w$. All terms belong to $\ell^2$, so bounded operator
composition and \eqref{eq:Neumann} give
\[
 \cR_{E,\kappa}f=\cR_{E,\kappa}g+D_Ew.
\]
The second term vanishes on $E\setminus E(\Omega)$. Therefore
\begin{align*}
 \mu_E\{|\cR_{E,\kappa}f|>\lambda\}
 &\le\mu_E(E(\Omega))+\mu_E\{|\cR_{E,\kappa}g|>\lambda\}\\
 &\le m(\Omega)+\lambda^{-2}\norm{\cR_{E,\kappa}g}_{2,\mu_E}^2\\
 &\le\lambda^{-1}\norm f_{1,m}
        +\lambda^{-2}\norm g_{2,m}^2
 \le\frac2\lambda\norm f_{1,m}.
\end{align*}
This bound is independent of $\kappa$.
Take $\kappa_j=1/j$. Pointwise convergence from
\eqref{eq:strong-limit} implies
\[
 \mathbf1_{\{|\cR_Ef|>\lambda\}}
 \le\liminf_{j\to\infty}
       \mathbf1_{\{|\cR_{E,\kappa_j}f|>\lambda\}}.
\]
Fatou's lemma, summed with the edge weights, proves
\eqref{eq:edge-main}. Neither the potentials $w$ nor their support sets
are required to converge as $\kappa\downarrow0$.
\end{proof}

\subsection{The vertex transform}
If $u=0$ on $V\setminus F$ and $x\notin N_1(F)$, then both $x$ and all
of its neighbors lie outside $F$. Thus
\begin{equation}\label{eq:vertex-support}
 D_Vu(x)=0\qquad(x\notin N_1(F)).
\end{equation}

\begin{proof}[Proof of the endpoint bound in \cref{thm:vertex}]
Again fix real $f\in\ell^1\cap\ell^2$ and $\lambda,\kappa>0$, and write
$f=g+A_\kappa w$ as in \cref{cor:signed}. In the output Hilbert space,
\[
 \cR_{V,\kappa}f=\cR_{V,\kappa}g+D_Vw.
\]
The second term is zero outside $N_1(\Omega)$. Using
\eqref{eq:H1-sets} and \eqref{eq:shifted}, we obtain
\begin{align*}
 m\{|\cR_{V,\kappa}f|_{\cH}>\lambda\}
 &\le m(N_1(\Omega))+m\{|\cR_{V,\kappa}g|_{\cH}>\lambda\}\\
 &\le H_1(G)m(\Omega)
      +\lambda^{-2}\norm{\cR_{V,\kappa}g}_{2,m;\cH}^2\\
 &\le\frac{H_1(G)}\lambda\norm f_{1,m}
      +\lambda^{-2}\norm g_{2,m}^2\\
 &\le\frac{H_1(G)+1}\lambda\norm f_{1,m}.
\end{align*}
Here $|F|_{\cH}$ denotes the function $x\mapsto|F(x)|_{\cH_x}$.
If $\Omega=\varnothing$ the same argument applies with
$N_1(\Omega)=\varnothing$. Otherwise $0<m(\Omega)<\infty$, as required
in \eqref{eq:H1-sets}.
Fiberwise convergence in \eqref{eq:strong-limit} and Fatou's lemma give
\eqref{eq:vertex-main}, exactly as in the edge case.
\end{proof}

\subsection{Endpoint extension and strong estimates}\label{sec:extension}
We give the details for the two different output spaces simultaneously.
Let $(Y,\nu;\mathcal K)$ be either the weighted edge set with scalar
fibers or $(V,m;\cH)$. Let $T$ be the corresponding transform on
$\ell^2(V,m)$, and let $C=2$ or $C=H_1(G)+1$, respectively.
We have proved, for real $h\in\ell^1\cap\ell^2$,
\begin{equation}\label{eq:weak-core}
 \nu\{|Th|>t\}\le\frac C t\norm h_{1,m},\qquad
 \norm{Th}_2\le\norm h_{2,m}.
\end{equation}

Write $\nu(y)=\nu(\{y\})>0$ for the weight of an output atom.
Fix $y$. For every $0<t<|Th(y)|$ the first inequality in
\eqref{eq:weak-core} gives $t\nu(y)\le C\norm h_{1,m}$.
Letting $t\uparrow|Th(y)|$, or treating $Th(y)=0$ separately, yields
\begin{equation}\label{eq:atom-bound}
 |Th(y)|\le\frac C{\nu(y)}\norm h_{1,m}.
\end{equation}
For real $f\in\ell^1(V,m)$, choose an increasing finite exhaustion
$F_j\uparrow V$ and set $f_j=f\mathbf1_{F_j}$. Then
$f_j\in c_{00}(V)$ and $f_j\to f$ in $\ell^1$.
Equation~\eqref{eq:atom-bound}, applied to differences, shows that
$Tf_j(y)$ is Cauchy in the complete fiber $\mathcal K_y$ for every $y$.
Define $T_1f(y)=\lim_jTf_j(y)$. The same bound for differences shows that
the limit is independent of the approximating sequence in
$\ell^1\cap\ell^2$, and that $T_1$ is linear. If
$f\in\ell^1\cap\ell^2$, then $f_j\to f$ also in $\ell^2$, so the
$\ell^2$ boundedness and coordinate evaluation identify $T_1f$ with $Tf$.
Pointwise convergence and Fatou's lemma give
\begin{equation}\label{eq:weak-extended}
 \nu\{|T_1f|>t\}\le\frac C t\norm f_{1,m}.
\end{equation}
Applying this estimate to differences proves continuity into weak
$\ell^1$. Any bounded weak-type extension has the coordinate continuity
\eqref{eq:atom-bound}, possibly with its own bound; density then proves
uniqueness.

For completeness, the strong estimate can be obtained directly from the
usual distribution-function proof of Marcinkiewicz interpolation
\cite{Stein1970}. For real $f\in c_{00}(V)$ and $t>0$, set
$f_{>t}=f\mathbf1_{\{|f|>t\}}$ and
$f_{\le t}=f\mathbf1_{\{|f|\le t\}}$.
Linearity of $T$ and the triangle inequality in each output fiber yield
\[
 \nu\{|Tf|>t\}
 \le\frac{2C}{t}\sum_{|f(x)|>t}|f(x)|m(x)
    +\frac4{t^2}\sum_{|f(x)|\le t}|f(x)|^2m(x).
\]
For $1<p<2$, the layer-cake formula and Tonelli's theorem therefore give
\begin{align}
 \norm{Tf}_{p,\nu;\mathcal K}^p
 &=p\int_0^\infty t^{p-1}\nu\{|Tf|>t\}\,dt\notag\\
 &\le\left(\frac{2Cp}{p-1}+\frac{4p}{2-p}\right)\norm f_{p,m}^p.
 \label{eq:explicit-interpolation}
\end{align}
Indeed, the two inner integrals are respectively
$\int_0^{|f(x)|}t^{p-2}dt=|f(x)|^{p-1}/(p-1)$ and
$\int_{|f(x)|}^{\infty}t^{p-3}dt=|f(x)|^{p-2}/(2-p)$ for
$f(x)\ne0$; zero values contribute zero before integration.
Density of $c_{00}(V)$ gives a unique bounded extension $T_p$ to
$\ell^p(V,m)$. For $f\in\ell^p\cap\ell^2$, the finite truncations
converge in both norms; their output limits agree at every atom, so
$T_p$ agrees with the spectral $\ell^2$ operator on this intersection.
For $f\in\ell^p\cap\ell^1$, the same argument identifies $T_p$ with
$T_1$. This verifies the consistency of all extensions without assuming
any inclusion between the different weighted $\ell^r$ spaces.
At $p=2$, use \eqref{eq:polar} directly instead of
\eqref{eq:explicit-interpolation}.

Finally, for $f=a+ib$ with real $a,b$, define $T_pf=T_pa+iT_pb$.
For $1<p<2$ the real estimate and $\norm a_p,\norm b_p\le\norm f_p$
give the complex estimate with at most twice the real strong-type bound.
It agrees on $\ell^p\cap\ell^2$ with the complex spectral transform.
At $p=2$ the norm remains at most $1$ by the complex energy identity.
This completes the extension and strong-type assertions of both theorems.

\section*{Acknowledgment of AI assistance}
The author used ChatGPT in this work. The author takes full responsibility for the content of the paper.

\end{document}